\documentclass[10pt,a4paper,twoside]{article}
\usepackage{amsthm,amsfonts,amssymb,amsmath,amscd,bezier,color}
\usepackage{graphicx}
\usepackage{hyperref}
\usepackage[all]{xy}

\newcommand{\Z}{\ensuremath{\mathbb{Z}}}

\newcommand{\RP}{\ensuremath{\mathbb{R}\mathrm{P}}}
\newcommand{\Pe}{\ensuremath{\mathcal{P}}}
\newcommand{\Qe}{\ensuremath{\mathcal{Q}}}

\newcommand{\q}{\ensuremath{\mathfrak{q}}}
\newcommand{\p}{\ensuremath{\mathfrak{p}}}
\newcommand{\A}{\ensuremath{\mathfrak{a}}}

\newtheorem{definition}{Definition}[section]
\newtheorem{proposition}[definition]{Proposition}
\newtheorem{lemma}[definition]{Lemma}
\newtheorem{remark}[definition]{Remark}
\newtheorem{theorem}[definition]{Theorem}

\begin{document}

\begin{center}

{\Large\sc Strongly surjective maps from certain two-complexes with trivial top-cohomology onto the projective plane} 

\vspace{8mm}
{\normalsize \sc Marcio C. Fenille \quad \& \quad Daciberg L. Gon\c calves}

\end{center}

\vspace{5mm}

\noindent{\bf Abstract:} For the model two-complex $K$ of the group presentation $\Pe=\langle x,y\,|\,x^{k+1}yxy \rangle$, with $k\geq1$ odd, we describe representatives for all free and based homotopy classes of maps from $K$ into the projective plane. As a result   we classify  the homotopy classes containing only surjective maps. With this approach we get an answer, for maps into the real projective plane, to a classical question in topological root theory, which is known so far, in dimension two, only for maps into the sphere, the torus and the Klein bottle. The answer follows by proving that for all $k\geq1$ odd, $H^2(K;\Z)=0$ and, for $k\geq3$ odd, there exist maps from $K$ into the real projective plane which are strongly surjective. For $k=1$, there does not exist  such a strongly surjective map. 

\vspace{2mm}


\vspace{2mm}



\section{Introduction and Main Theorem}

\hspace{4mm} The Hopf-Whitney Classification Theorem \cite[Corollary 6.19, p.\,244]{Whitehead} implies that, for a finite and connected $n$-dimensional complex $X$ (an $n$-complex, for short), the set $[X;S^n]$ of the free homotopy classes of maps from $X$ into the $n$-sphere $S^n$ is in a one-to-one correspondence with the integer cohomology group $H^n(X;\Z)$. Thus, there exists a map from $X$ onto $S^n$ whose free homotopy class contains only surjective maps if and only if $H^n(X;\Z)\neq0$. Such a map is called a {\it strong surjection} or a {\it strongly surjective map}.

The composition of a strong surjection from $X$ onto $S^n$ with the double covering map $\p:S^n\to\RP^n$ provides a strong surjection from $X$ onto the $n$-dimensional projective space $\RP^n$. Hence, the assumption $H^n(X;\Z)\neq0$ implies the existence of a strong surjection from $X$ onto $\RP^n$. In this article, we prove, for $n=2$, that the converse implication does not hold true. Namely, we show the existence of strongly surjective maps from $K$ onto $\RP^2$ for a certain two-complex $K$ for which $H^2(K;\Z)=0$.

This work concerns an important and interesting problem in topological root theory, namely, to know for what closed $n$-manifold $Y$, the nullity of the top integer cohomology group of a $n$-complex $X$ forces the non-existence of strong surjections from $X$ into $Y$. 

Besides the relationship with the Hopf-Whitney Classification Theorem, the relevance of the problem lies in the fact that, by the Universal Coefficient Theorem for Cohomology \cite[Theorem 3.2, p.\,195]{Hatcher}, an $n$-complex $X$ with $H^n(X;\Z)=0$ is (co)homologically like a $(n-1)$-complex, since such a nullity is equivalent to $H_n(X;\Z)=0$ and $H_{n-1}(X;\Z)$ torsion free; that is, the invariant $H^n(\,\cdot\,;\Z)$ is not able to detect the existence of $n$-cells even when the inclusion $X^{n-1}\hookrightarrow X$ of the $(n-1)$-skeleton of $X$ into $X$ is not a homotopy equivalence.

As a consequence of the classification theorem for surfaces, in dimension two it is more feasible to completely solve the problem. However, the first contributions \cite{Aniz-2006, Aniz-2008} were presented in dimension three. In the 2000's, Claudemir Aniz answers the problem (proposed by Daciberg Lima Gon\c calves) for the following 3-manifolds: the cartesian product $S^1\times S^2$, the non-orientable $S^1$-bundle over $S^2$ and the orbit space of $S^3$ with respect to the action of the Quaternion group. He showed that only in the second case there exists a 3-complex $X$ with $H^3(X;\Z)=0$ and a strong surjection from $X$ onto the corresponding $3$-manifold.

The first conclusive answer in dimension two, other than that provide by the Hopf-Whitney Classification Theorem, was present in 2016 in \cite{Fenille-Top-Cohomology}, in which the first author built a countable collection of two-complexes with trivial second integer cohomology group and, from each of them, there exists a strong surjection onto the torus $S^1\times S^1$. By composing each  such strong surjection with the double covering map from the torus onto the Klein bottle, we get a strong surjection onto the Klein bottle.

There is no other known conclusive answer to the problem. Therefore, in dimension two, there are answers only for maps into the sphere, the torus and the Klein bottle.

However, there exist partial answers for maps into the projective plane $\RP^2$. We refer to the results presented in \cite{Fenille-One-Relator-RP2,Fenille-RP2-Square}. In fact, in \cite{Fenille-One-Relator-RP2} the germ of the conclusive answer is present: it is shown that a cohomological condition implies the non existence of a strongly surjective map. Here, we give a more complete answer as a consequence of the following theorem.

\begin{theorem}[Main Theorem]\label{Main-Theorem}
Let $K$ be the model two-complex of the group presentation $\Pe=\langle x,y\,|\,x^{k+1}yxy \rangle$, with $k\geq1$ odd. Then $H^2(K;\Z)=0$ and we have:
    \begin{itemize}
    \item[{\rm (1)}] If $k=1$, then $[K;\RP^2]\equiv [K;\RP^2]^{\ast}\equiv\{1\}\sqcup\{\bar{0}\}$ and both the homotopy classes contain non-surjective maps.
    \item[{\rm (2)}] If $k=2p-1\geq3$, then $[K;\RP^2]^{\ast}\equiv\{1\}\sqcup\Z_k$ and $[K;\RP^2]\equiv\{1\}\sqcup\Z_{p}$. The free homotopy classes corresponding to $1$ and $\bar{0}$ contain non-surjective maps and the remaining $p-1$ classes contain only surjective maps.
    \end{itemize}
\end{theorem}

The notation used in Theorem \ref{Main-Theorem} is detailed in the text. We anticipate that the symbol $\equiv$ indicates bijection between sets (without preserving any algebraic structure). 

We describe the structure of this article, highlighting the steps of the proof of the Main Theorem. In Section \ref{Section-Action} we introduce notations and recall some results regarding  the action of the fundamental group over based homotopy classes. In Section \ref{Section-Self-Maps} we describe in detail the free and the based homotopy classes of self-maps of the projective plane and we prove that the action of $\pi_1(\RP^2)$ on the set $[\RP^2;\RP^2]^{\ast}_{id}$ exchanges based homotopy classes of maps of opposite {\it twisted degree}. In Section \ref{Section-Complex-K} we finally consider the model two-complex $K$ of the group presentation $\Pe=\langle x,y\,|\,x^{k+1}yxy \rangle$, with $k\geq1$ odd, and we prove that, for the unique twisted integer coefficient system $\beta$ over $K$, other than the trivial one, the corresponding twisted cohomology group $H^2(K;_{\beta}\!\Z)$ is cyclic of order $k$. In Section \ref{Section-Key-Map} we build a special map $\omega: K\to\RP^2$ for which the induced homomorphism on twisted cohomology groups, namely $\omega^{\ast}:H^2(\RP^2;_{\varrho}\!\Z)\to H^2(K;_{\beta}\!\Z)$, corresponds to the natural epimorphism $\Z\to\Z/k\Z$, and so $\omega$ is strongly surjective, for $k\neq 1$. Section \ref{Section-Proof}    consists of the proof of Theorem 1.1. The proof follows from a complete description of representatives for all the free and based homotopy classes of maps from $K$ into $\RP^2$. The main step is the proof that  each based map $f:K\to\RP^2$ inducing the homomorphism $\beta$ on fundamental groups is based homotopic to a map $f_n=h_n\circ\omega$, in which $\omega$ is the special map built in Section \ref{Section-Key-Map} and $h_n:\RP^2\to\RP^2$ is a map of twisted degree $n$, where $n$ is an odd integer in the set $\{-k,-k+1 \ldots, k-1,k\}$. Consequently, the action of $\pi_1(K)$ on $[K;\RP^2]^{\ast}_{\beta}$ can be obtained from the action of $\pi_1(\RP^2)$ on $[\RP^2;\RP^2]^{\ast}_{id}$ and the induced homomorphism on twisted cohomology groups by $f_n$ is not trivial for $n\neq\pm k$. This then forces   $f_n$ to be strongly surjective.

Throughout the text, for the sake of simplicity, we call a finite and connected two-dimensional $CW$-complex by a two-complex. We also simplify $f$ is a continuous map by $f$ is a map. Further, we consider the cyclic group $\Z_2=\{1,-1\}$ with its multiplicative structure  and, where appropriated, we identify an automorphism $\tau\in{\rm Aut}(\Z)$ with its value $\tau(1)$.




\section{Actions of $\pi_1$ on based homotopy classes}\label{Section-Action}

\hspace{4mm} Let $K$ be a two-complex with fundamental group $\Pi=\pi_1(K)$ and take a 0-cell $e^0$ in $K$ to be its base-point. Consider the real projective plane $\RP^2$ with its minimal cellular structure, namely $\RP^2=c^0\cup c^1\cup c^2$, and take $c^0$ to be the base-point.

In what follows, we distinguish free homotopies and based homotopies starting at a given based map $f:K\rightarrow\RP^2$. We observe that, by the Cellular Approximation Theorem, each map from $K$ into $\RP^2$ is freely homotopic to a based map. Hence, in order to study free or based homotopy classes, we can assume that a homotopy class always admite a representative   given {\it a priori} by a map which is based. We define:
    \begin{itemize}
    \item $[K;\RP^2]$ is the set of free homotopy classes $[f]$ of maps $f:K\rightarrow\RP^2$.
    \item $[K;\RP^2]^{\ast}$ is the set of based homotopy classes $[f]^{\ast}$ of based maps $f:K\rightarrow\RP^2$.
    \item $[K;\RP^2]^{\ast}_{\alpha}$ is the set of based homotopy classes $[f]^{\ast}$ of based maps $f:K\rightarrow\RP^2$ such that $\alpha=f_{\#}:\pi_1(K)\rightarrow\pi_1(\RP^2)$.
    \end{itemize}

It follows that $$[K;\RP^2]^{\ast}=\bigsqcup_{\alpha\in\hom(\Pi;\Z_2)}[K;\RP^2]^{\ast}_{\alpha}.$$

Of course, we are identifying, in this description, $\hom(\Pi;\Z_2)$ with $\hom(\pi_1(K);\pi_1(\RP^2))$.

The fundamental group $\pi_1(\RP^2)$ acts on the set $[K;\RP^2]^{\ast}$ and, following \cite[Chapter V, Corollary 4.4]{Whitehead}, $[K;\RP^2]$ corresponds to the quotient set of $[K;\RP^2]^{\ast}$ by this action, what we indicate by $$[K;\RP^2]\equiv\frac{[K;\RP^2]^{\ast}}{\pi_1(\RP^2)}.$$

We recall, in a general context, how the action of $\pi_1(Y)$ on $[X;Y]^{\ast}$ is defined. Consider based spaces $(X,x_0)$ and $(Y,y_0)$. Let $f_0,f_1:X\rightarrow Y$ be based maps and let $u:I\rightarrow Y$ be a loop in $Y$ based at $y_0$. Suppose there exists a homotopy $F:X\times I\rightarrow Y$, starting at $f_0$ and ending at $f_1$, such that $F(x_0,t)=u(t)$. Then we say that $f_0$ is freely homotopic to $f_1$ along to $u$ and we write $f_0\simeq_{u}f_1$. If $u$ is the constant path at the base-point $y_0$, we say that $f_0$ is based homotopic to $f_1$ and we write $f_0\simeq_{\ast}f_1$. We have:
    \begin{itemize}
    \item[{\rm (1)}] Given a based map $f_0:X\rightarrow Y$ and a loop $u$ in $Y$ based at $y_0$, then $f_0\simeq_{u} f_1$ for some based map $f_1:X\rightarrow Y$.
    \item[{\rm (2)}] If $f_0\simeq_{u}f_1$ and $f_0\simeq_{v}f_2$ and $u\simeq v\,({\rm rel.}\partial I)$, then $f_1\simeq_{\ast}f_2$.
    \item[{\rm (3)}] If $f_0\simeq_{u}f_1$ and $f_1\simeq_{v}f_2$, then $f_0\simeq_{uv}f_2$.
    \end{itemize}

This defines the action of $\pi_1(Y)$ on $[X;Y]^{\ast}$. Thus: given a based map $f_0:X\rightarrow Y$ and an element $[u]\in\pi_1(Y)$ represented by a loop $u$ in $Y$ based at $y_0$, there exists a based map $f_1:X\rightarrow Y$ such that $f_0\simeq_{u}f_1$, and we define the action of $[u]$ on $[f_0]^{\ast}$ to be $[f_1]^{\ast}$, that is, $$[u][f_0]^{\ast}=[f_1]^{\ast}.$$

Returning to our approach, let $\sigma:I\rightarrow\RP^2$ be the loop in $\RP^2$ based at $c^0$ whose trajectory encircles once the 1-cell $c^1$. Then $[\sigma]\in\pi_1(\RP^2)$ is the generator of $\pi_1(\RP^2)$. Furthermore, $\sigma$ induces the identity automorphism $$\hat{\sigma}:\pi_1(\RP^2)\rightarrow\pi_1(\RP^2) \quad{\rm given \ by}\quad \hat{\sigma}([u])=[\sigma^{-1}][u][\sigma]=[u].$$

\begin{lemma}\label{Lemma-Invariance-Action}
Each subset $[K;\RP^2]^{\ast}_{\alpha}$ of $[K;\RP^2]^{\ast}$ is invariant by the action of $\pi_1(\RP^2)$.
\end{lemma}
\begin{proof}
Consider the generator $[\sigma]$ of $\pi_1(\RP^2)$. Given a based homotopy class $[f_0]^{\ast}\in[K;\RP^2]^{\ast}_{\alpha}$, we take a based map $f_1:K\rightarrow\RP^2$ such that $f_0\simeq_{\sigma}\!f_1$, that is, there exists a homotopy $H:f_0\simeq f_1$ such that $H(e^0\!,t)=\sigma(t)$. Then $(f_{1})_{\#}=\hat{\sigma}\circ (f_{0})_{\#}=id\circ\alpha=\alpha$ and, by definition, $[\sigma][f_0]^{\ast}=[f_1]^{\ast}$.
\end{proof}

It follows that $$[K;\RP^2]\equiv\bigsqcup_{\alpha\in\hom(\Pi;\Z_2)}\frac{ \ [K;\RP^2]^{\ast}_{\alpha} \ }{\pi_1(\RP^2)}.$$

In the case in which $K$ is {\it aspherical} (has contractible universal covering), Theorem 4.12 of \cite{Livro-Verde-Chapter-II} provides, for each $\alpha\in\hom(\Pi;\Z_2)$, a bijection $$[K;\RP^2]_{\alpha}^{\ast}\equiv H^2(K;_{\alpha}\!\Z),$$ in which $H^2(K;_{\alpha}\!\Z)$ is the second cohomology group of $K$ with the local integer coefficient system $\alpha:\Pi\to\Z_2\approx{\rm Aut}(\Z)$. We explore this fact next.


\section{Self-maps of the projective plane}\label{Section-Self-Maps}

\hspace{4mm} In this section we present an analysis of the free and based homotopy classes of self-maps of the real projective plane. In a sense, what we present explains certain facts that can be inferred from \cite[Proposition 2.1]{Daciberg-Spreafico}.

Throughout the section, $\p:S^2\to\RP^2$ is the double covering map, $\A:S^2\to S^2$ is the antipodal map, $\Z^{odd}$ is the set of the odd integers and $\Z^{odd}_{+}$ is the set of the non-negative ones.

We consider the sphere $S^2$ as the suspension of $S^1$, that is, the quotient space obtained from the cylinder $S^1\times[-1,1]$ by collapsing $S^1\times\{-1\}$ to a single point (the south pole) and $S^1\times\{1\}$ to another single point (the north pole). Thus, we can  write a point of $S^2$ as a class $\lceil e^{i\theta},\tau\rceil$ in which $(e^{i\theta},\tau)\in S^1\times[-1,1]$. We take:
    \begin{align*}
    s^0_1 & =\lceil 1,0\rceil \ \textrm{to be the base-point in} \ S^2; \\ s^0_2 & =-s^0_1=\lceil -1,0\rceil \ \textrm{to be the antipodal point of} \ s^0_1; \\ s^1_1 & = \ \textrm{the half-equator arc}  \ \lceil e^{i\theta},0\rceil, \ \textrm{for} \ 0\leq\theta\leq\pi, \ \textrm{from} \ s^0_1 \ \textrm{to} \ s^0_2. \\ \sigma & =\p(s^1_1) \ \textrm{to be the loop representing the generator of} \ \pi_1(\RP^2).
    \end{align*}

The orientation of a loop provides over $\RP^2$ the local integer coefficient system $$\varrho:\pi_1(\RP^2)\rightarrow{\rm Aut}(\Z) \ \ {\rm given \ by} \ \ \varrho(1)=1 \ \, {\rm and} \ \, \varrho(-1)=-1.$$

Next, we consider the cohomology group $H^2(\RP^2;_{\varrho}\!\Z)$ with the local integer coefficient system $\varrho$.

We stablish the following one-to-one correspondences: $$[\RP^2;\RP^2]^{\ast}\equiv\Z_2\sqcup\Z^{odd} \qquad{\rm and}\qquad [\RP^2;\RP^2]\equiv\Z_2\sqcup\Z^{odd}_{+}.$$ 

All maps given {\it a priori} will be considered to be based. 

Firstly, we write $$[\RP^2;\RP^2]^{\ast}=[\RP^2;\RP^2]^{\ast}_0\sqcup[\RP^2;\RP^2]^{\ast}_{id},$$ in which the subscripts $0$ and $id$ indicate that the corresponding maps induce the trivial and the identity homomorphism on fundamental groups, respectively.

It follows by Lemma \ref{Lemma-Invariance-Action} that $$[\RP^2;\RP^2]\equiv\frac{[\RP^2;\RP^2]^{\ast}_0}{\pi_1(\RP^2)}\sqcup\frac{[\RP^2;\RP^2]^{\ast}_{id}}{\pi_1(\RP^2)}.$$

Let $h:\RP^2\to\RP^2$ be a based map and take $\widetilde{h}:S^2\to S^2$ to be the based lifting of $h\circ\p$  through $\p:S^2\to\RP^2$, so that we have the commutative diagram:
    \begin{center}
    \begin{tabular}{c}\xymatrix{ S^2 \ar[r]^-{\widetilde{h}} \ar[d]_-{\p} & S^2 \ar[d]^-{\p} \\ \RP^2 \ar[r]^{h} & \RP^2 \\ }
    \end{tabular}
    \end{center}

We claim that $\widetilde{h}$ is necessarily either even or odd; in fact, for each $x\in S^2$, we have either $\widetilde{h}(-x)=\widetilde{h}(x)$ or $\widetilde{h}(-x)=-\widetilde{h}(x)$, and so $\langle \widetilde{h}(x),\widetilde{h}(-x)\rangle=\pm1$. By continuity, the map $x\mapsto\langle\widetilde{h}(x),\widetilde{h}(-x)\rangle$ is constant equal to either $1$ or $-1$, and so $\widetilde{h}$ is either even or odd.

If $h:\RP^2\to\RP^2$ induces the trivial homomorphism on fundamental groups, that is, $[h]^{\ast}\in[\RP^2;\RP^2]^{\ast}_0$, then $h$ lifts through $\p$ to a based map $\overline{h}:\RP^2\rightarrow S^2$. Now, $$[\RP^2;S^2]^{\ast}\equiv H^2(\RP^2;\Z)\approx\Z_2,$$ and we can describe the two classes $[\overline{h}_{00}]^{\ast}$ and $[\overline{h}_{01}]^{\ast}$ by means of its representing maps, namely, $\overline{h}_{00}:\RP^2\to S^2$ is the constant map and $\overline{h}_{01}:\RP^2\to S^2$ is the quotient map that collapses the one-skeleton $S^1\subset\RP^2$ to the base-point of $S^2$. Defining the composed maps $h_{0i}=\p\circ\overline{h}_{0i}:\RP^2\to\RP^2$ for $i=0,1$, we have $$[\RP^2;\RP^2]^{\ast}_0=\{[h_{00}]^{\ast},[h_{01}]^{\ast}\}\equiv\Z_2.$$ 

Since $h_{00}$ and $h_{01}$ lift through $\p$ and obviously $[\RP^2;S^2]^{\ast}\equiv[\RP^2;S^2]$, it follows that$$[\RP^2;\RP^2]_0\equiv[\RP^2;\RP^2]^{\ast}_0\equiv\Z_2.$$ 

We remark that both the liftings $\widetilde{h}_{00}=\overline{h}_{00}\circ\p\,(=constant)$ and $\widetilde{h}_{01}=\overline{h}_{01}\circ\p$ are even self-maps of $S^2$. Now, if $\widetilde{h}:S^2\to S^2$ is even, then $\widetilde{h}=\widetilde{h}\circ\A$, and so $\deg(\widetilde{h})=-\deg(\widetilde{h})$, which forces $\deg(\widetilde{h})=0$ and, therefore, $\widetilde{h}$ is homotopically trivial, which does not imply that $h$ is itself homotopically trivial. This is what happens with the map $\widetilde{h}_1$, that is, $\widetilde{h}_1$ is homotopic to the constant map, but the maps $\overline{h}_1$ and $h_1$ are not.

On the other hand, it follows from Borsuk-Ulam Theorem (in its version presented in \cite[Chapter 2, \S\,6, p.\,91]{Guillemin-Pollack}) that if $\widetilde{h}:S^2\to S^2$ is odd, then $\deg(\widetilde{h})$ is odd. By the way, maps $\widetilde{h}:S^2\to S^2$ of arbitrary odd degree they do exist: for, given an odd integer $k$, the suspension $\widetilde{h}_k:S^2\to S^2$ of the map $S^1\ni z\mapsto z^k\in S^1$ is odd and has degree $k$.

Each such an odd map $\widetilde{h}_k:S^2\to S^2$ induces on the  quotient a based map $h_k:\RP^2\to\RP^2$, and it is easy to see that $h_k$ induces the identity homomorphism of fundamental groups, because the map $\p\circ\widetilde{h}_k$ maps the 1-cell $s^1_1$ in $S^2$ onto $k$ times the 1-cell $c^1$ in $\RP^2$. Thus, for each odd $k$, we have $[h_k]^{\ast}\in[\RP^2;\RP^2]_{id}^{\ast}$.

Since the degree classifies the homotopy classes of self-maps of $S^2$, it follows from the Lifting Homotopy Property that for two odd integers $k\neq l$, the corresponding maps $h_k,h_l:\RP^2\to\RP^2$ are not based homotopic. 

Therefore, the function $[h]^{\ast}\mapsto\deg(\widetilde{h})$ provides a one-to-one correspondence $$[\RP^2;\RP^2]_{id}^{\ast}\equiv\Z^{odd}.$$

Equivalently, this bijection can be written as $[h]^{\ast}\mapsto d_{\varrho}(h)$, in which the last number is the {\it twisted degree} of $h$, that is, the integer $d_{\varrho}(h)$ such that the homomorphism $$h^{\ast}:H^2(\RP^2;_{\varrho}\!\Z)\to H^2(\RP^2;_{\varrho}\!\Z),$$ induced by $h$ on cohomology groups with the non-trivial local integer coefficient system $\varrho$, corresponds to the multiplication by $d_{\varrho}(h)$. This will be clearer after Section \ref{Section-Key-Map}.

Now, for each odd $k>0$, both the maps $\widetilde{h}_{-k}$ and $\A\circ\widetilde{h}_k$ have the same degree $-k$, and so they are freely homotopic, but not based homotopic, since $\A\circ\widetilde{h}_k$ is not even based.

We present a special free homotopy $\widetilde{H}:\tilde{h}_{-k}\simeq\A\circ\widetilde{h}_k$. We define $\widetilde{H}:S^2\times I\to S^2$ by $$\widetilde{H}\big(\lceil e^{i\theta},\tau\rceil,t\big)=\big\lceil\mathbf{r}(t\pi)\!\cdot\!e^{i(2t-1)k\theta},(1-2t)\tau\big\rceil,$$ in which $z\mapsto\textbf{r}(t\pi)\!\cdot\!z$ is the positive rotation under angle $t\pi$ in the complex plane. We have:
    \begin{align*}
    \widetilde{H}\big(\lceil e^{i\theta},\tau\rceil,0\big) & =\lceil e^{-ik\theta},\tau\rceil=\widetilde{h}_{-k}\big(\lceil e^{i\theta},\tau\rceil\big), \\ \widetilde{H}\big(\lceil e^{i\theta},\tau\rceil,1\big) & =\lceil -e^{ik\theta},-\tau\rceil=-\widetilde{h}_{k}\big(\lceil e^{i\theta},\tau\rceil\big)=\A\circ \widetilde{h}_k(\lceil e^{i\theta},\tau\rceil\big).
    \end{align*}

Hence, $\widetilde{H}$ is really a free homotopy starting at $\widetilde{h}_{-k}$ and ending at $\A\circ\widetilde{h}_{k}$. Such  a homotopy is not based, since the trajectory of the path $t\mapsto\widetilde{H}(\lceil 1,0\rceil,t)$ is the half-equator arc $s^1_1$. 

Now, we observe that, since $k$ is odd, $\widetilde{H}$ is odd in the first coordinate, that is, $$\widetilde{H}\big(\!-\lceil e^{i\theta},\tau\rceil,t\big)=-\widetilde{H}\big(\lceil e^{i\theta},\tau\rceil,t\big).$$

Thus, $\widetilde{H}$ induces to quotient a free homotopy $H:\RP^2\times I\to\RP^2$ starting at $h_{-k}$ and ending at $h_k$ (since $h_k=\p(\A\circ\widetilde{h}_k)$). Moreover, the trajectory of the path $t\mapsto H(c^0,t)$ is the loop $\sigma$ whose path homotopy class is the generator of $\pi_1(\RP^2)$. 

We have proved that: the action of $\pi_1(\RP^2)$ on $[\RP^2;\RP^2]^{\ast}_{id}$ exchanges the based homotopy classes $[h_k]^{\ast}$ and $[h_{-k}]^{\ast}$ and so the function $[h]\mapsto|d_{\varrho}(h)|$ provides a bijection $$[\RP^2;\RP^2]_{id}\equiv\Z^{odd}_{+}.$$


\section{The model two-complex of $\Pe=\langle x,y\,|\,x^{k+1}yxy \rangle$}\label{Section-Complex-K}

\hspace{4mm} Let $K$ be the model two-complex of the group presentation $\Pe=\langle x,y\,|\,x^{k+1}yxy\rangle$, with $k\geq1$ odd, that is, the two-complex with a single $0$-cell $e^0$, two 1-cells $e^1_x\cup e^1_y$ and a single two-cell $e^2$ which is attached on the one-skeleton $K^1=e^0\cup e^1_x\cup e^1_y$ by spelling the word $r=x^{k+1}yxy$. We take the $0$-cell $e^0$ to be the base-point of $K$.

The fundamental group of $K$ is the group $\Pi=F(x,y)/N(r)$ presented by $\Pe$. Let  $\bar{x}$ and $\bar{y}$ in $\Pi$ be the images of $x$ and $y$, respectively, by the natural homomorphism $F(x,y)\to\Pi$ from the free group $F(x,y)$ onto $\Pi$.

In what follows, we consider the cohomology groups of $K$ with local integer coefficient systems, which we call {\it twisted cohomology groups}, for  short.

Since the group $\Pi$ has two generators, $\bar{x}$ and $\bar{y}$, and in the word $r=x^{k+1}yxy$ the sums of the powers of the letters $x$ and $y$ are respectively $k+2$ (which is odd) and $2$, we have just one local integer coefficient systems over $K$, other than the trivial one, namely, the system $$\beta:\Pi\rightarrow{\rm Aut}(\Z) \quad \textrm{given by} \quad \beta(\bar{x})=1 \quad {\rm and}\quad \beta(\bar{y})=-1.$$

\begin{proposition}
Let $K$ be the model two-complex of the presentation $\Pe=\langle x,y\,|\,x^{k+1}yxy\rangle$, with $k\geq1$ odd. We have:
    \begin{itemize}
    \item[{\rm (1)}] $H^2(K;\Z)=0$ and $H^2(K,_{\beta}\!\Z)\approx\Z/k\Z$.
    \item[{\rm (2)}] $K$ is aspherical.
    \end{itemize}
\end{proposition}
\begin{proof}
The first statement of (1) follows from a straightforward analysis of the cellular co-chain complex of $K$ and the second one is announciated in \cite[Example 7.3]{Fenille-One-Relator-RP2}, but without details. Since next we need to  identify explicitly a generator of $H^2(K;_{\beta}\!\Z)$, we provide a detailed calculation of the group $H^2(K;_{\beta}\!\Z)$ 
after this proof. Assertion (2) follows from (1) and  \cite[Proposition 4.1]{Fenille-One-Relator-RP2}.
\end{proof}

\begin{remark}{\rm
Before proceeding to  the calculations of the twisted cohomology group $H^2(K,_{\beta}\!\Z)$, we observe that, for $k\ne 0$ even, the two-complex $K$ is also aspherical. This fact follows from \cite[Section 4]{Fenille-One-Relator-RP2}, in which it is remarked that a one-relator model two-complex is aspherical if and only if the single relator of its presentation is not freely trivial and has period one. Thus, in order to have $K$ non-aspherical we should take $k=0$. On the other hand, if $k\geq1$ is even, then $H^2(K;\Z)\approx\Z_2$ and so $K$ would not be an interesting two-complex from the viewpoint of the inspiring problem of this article.}
\end{remark}

Returning to the case $k\geq1$ odd, we compute the twisted cohomology group of $H^2(K,_{\beta}\!\Z)$. We use the procedure and the notations presented in \cite[Section 3]{Fenille-One-Relator-RP2}. Briefly:
\begin{itemize}
\item $\xi_{\beta}:\Z[\Pi]\to\Z$ is the $\beta$-augmentation function $\xi_{\beta}(\sum_kn_i\pi_i)=\sum_in_i\beta(\pi_i)$.
\item $\|\cdot\|:\Z[F(x,y)]\rightarrow\Z[\Pi]$ is the natural extension on group rings of the natural homomorphism $F(x,y)\rightarrow\Pi=F(x,y)/N(r)$.
\item $\dfrac{\partial}{\partial x},\dfrac{\partial}{\partial y}:F(x,y)\rightarrow\Z[F(x,y)]$ are the Reidmeister-Fox derivatives.
\end{itemize}

For the relator word $r=x^{k+1}yxy$, we have:
    \begin{align*}
    \frac{\partial r}{\partial x} & = (1+x+\cdots+x^k)+x^{k+1}y \quad\textrm{and so}\quad  \xi_{\beta}\big(\|\frac{\partial r}{\partial x}\|\big)=(k+1)-1=k,  \\ \noalign{\smallskip} \frac{\partial r}{\partial y} & = x^{k+1}(1+yx) \quad\textrm{and so}\quad \xi_{\beta}\big(\|\frac{\partial r}{\partial y}\|\big)=1(1-1)=0.
    \end{align*}

Consider the cellular chains of $K$ with its natural identifications and generators:
    \begin{align*}
    C_0(K) & =H_0(K^0)\approx\Z\langle e^0\rangle, \\ \noalign{\smallskip} C_1(K) & =H_1(K^1\!,K_{\Pe}^0)\approx\Z^2\langle e^1_x,e^1_y\rangle, \\ \noalign{\smallskip} C_2(K) & =H_2(K,K^1)\approx\Z\langle e^2\rangle.
    \end{align*}

Let $\widetilde{K}$ be the universal covering space of $K$, endowed with its natural cellular structure. Select a $0$-cell $\tilde{e}^0$ over $e^0$\!, a $1$-cell $\tilde{e}^1_x$ over $e^1_x$, a $1$-cell $\tilde{e}^1_y$ over $e^1_y$ and a $2$-cell $\tilde{e}^2$ over $e^2$. The group $\Pi$ acts on the left (via covering transformation) on the cellular chain complex $C_q(\widetilde{K})=H_q(\widetilde{K}^q\!,\widetilde{K}^{q-1})$ making it into a left $\Z[\Pi]$-module, so that we have identifications
    \begin{align*}
    C_0(\widetilde{K}) & =\Z[\Pi]\langle\tilde{e}^0 \rangle, \\ \noalign{\smallskip} C_1(\widetilde{K}) & =\Z[\Pi]^2\langle\tilde{e}^1_x,\tilde{e}^1_y \rangle, \\ \noalign{\smallskip} C_2(\widetilde{K}) & =\Z[\Pi]\langle\tilde{e}^2\rangle.
    \end{align*}

Via this identifications and considering the action $\beta:\Pi\rightarrow{\rm Aut}(\Z)$, we have the corresponding twisted cellular chain complex of left $\Z[\Pi]$-modules $$C_{\ast}^{\beta}(\widetilde{K}):0\rightarrow C_2(\widetilde{K})\stackrel{\tilde{\partial}_2^{\beta}}{\longrightarrow}C_1(\widetilde{K})\stackrel{\tilde{\partial}_1^{\beta}}{\longrightarrow}C_0(\widetilde{K})\rightarrow0,$$ in which the boundaries operators are given by
    \begin{align*}
    \tilde{\partial}_1^{\beta}(\tilde{e}^1_x) & = \xi_{\beta}(1-\bar{x})\tilde{e}^0=0, \\ \noalign{\smallskip} \tilde{\partial}_1^{\beta}(\tilde{e}^1_y) & = \xi_{\beta}(1-\bar{y})\tilde{e}^0=2\tilde{e}^0, \\ \noalign{\smallskip} \tilde{\partial}_2^{\beta}(\tilde{e}^2) & = \xi_{\beta}\big(\|\frac{\partial r}{\partial x}\|\big)\tilde{e}^1_x+\xi_{\beta}\big(\|\frac{\partial r}{\partial y}\|\big)\tilde{e}^1_y = k\tilde{e}^1_x.
    \end{align*}

Consider the corresponding twisted cellular co-chain complex $$C^{\ast}_{\beta}(\widetilde{K}):0\leftarrow\hom^{\Pi}(C_2(\widetilde{K});\Z)\stackrel{\tilde{\delta}_2^{\beta}}{\longleftarrow}\hom^{\Pi}(C_1(\widetilde{K});\Z)\stackrel{\tilde{\delta}_1^{\beta}}{\longleftarrow}\hom^{\Pi}(C_0(\widetilde{K});\Z)\leftarrow0.$$

In each $\hom^{\Pi}(C_i(\widetilde{K});\Z)$, the integers $\Z$ is seen as a left $\Z[\Pi]$-module via the action $\beta:\Pi\rightarrow{\rm Aut}(\Z)$. The co-boundaries operators $\tilde{\delta}_{\ast}^{\beta}$ are defined by the usual dual form $$\tilde{\delta}_{\ast}^{\beta}(\phi)=\phi\circ\tilde{\partial}_{2}^{\beta}.$$

Explicitly, a given co-chain $\phi\in\hom^{\Pi}(C_1(\widetilde{K});\Z)$ is defined by its values $\phi(\tilde{e}^1_x)$ and $\phi(\tilde{e}^1_x)$, and the co-chain $\tilde{\delta}_2^{\beta}(\phi):C_2(K)\to\Z$ is given by $$\tilde{\delta}_2^{\beta}(\phi)(\tilde{e}^2)=\xi_{\beta}\Big(\|\frac{\partial r}{\partial x}\|\Big)\phi(\tilde{e}^1_x)+\xi_{\beta}\Big(\|\frac{\partial r}{\partial y}\|\Big)\phi(\tilde{e}^1_y)= k\phi(\tilde{e}^1_x).$$

Now, $\hom^{\Pi}(C_2(\widetilde{K});\Z)\approx\Z$ is generated by the co-chain $\phi_2^{\ast}:C_2(\widetilde{K})\to\Z$ given by $$\phi_2^{\ast}(\tilde{e}^2)=1, \ \textrm{that is, the dual of the chain }\tilde{e}^2.$$

Analogously, $\hom^{\Pi}(C_1(\widetilde{K});\Z)\approx\Z^2$ is generated by the co-chains $\phi_x^{\ast},\phi_y^{\ast}:C_1(\widetilde{K})\to\Z$ which are the dual of the chains $\tilde{e}^1_x$ and $\tilde{e}^1_y$, that is, $$\phi_x^{\ast}(\tilde{e}^1_x)=1, \ \phi_x^{\ast}(\tilde{e}^1_y)=0,\quad{\rm and}\quad \phi_y^{\ast}(\tilde{e}^1_x)=0, \ \phi_y^{\ast}(\tilde{e}^1_y)=1.$$

Thus, the co-boundary operator $\tilde{\delta}_2^{\beta}:\hom^{\Pi}(C_1(\widetilde{K});\Z)\to\hom^{\Pi}(C_2(\widetilde{K});\Z)$ is given by
$$\tilde{\delta}_2^{\beta}(\phi_x^{\ast}) = k\phi_2^{\ast} \quad{\rm and}\quad \tilde{\delta}_2^{\beta}(\phi_y^{\ast})=0.$$

Therefore,
$$ H^2(K;_{\beta}\!\Z)\approx\dfrac{\hom^{\Pi}(C_2(\widetilde{K});\Z)}{{\rm Im}(\tilde{\delta}_2^{\beta})}\approx \dfrac{\Z\langle \phi_2^{\ast}\rangle}{\langle k\phi_2^{\ast} \rangle}\approx\frac{\Z}{k\Z}\big\langle \phi_2^{\ast}+\langle k\phi_2^{\ast}\rangle\big\rangle.$$

\begin{remark}{\rm  Let us point out that the group $H^2(K;_{\beta}\!\Z)$ depends on the word $r$, and not only on the sums of the powers of the letters $x$ and $y$. For example, if we take $r=x^{k+2+n}y^2x^{-n}$, for $k\geq1$ and $n\geq 0$, it can be shown that $H^2(K;_{\beta}\!\Z)\approx \Z_{k+2}$.}
\end{remark}


\section{Maps from $K$ into $\RP^2$}\label{Section-Key-Map}

\hspace{4mm} In this section, we continue to consider the model two-complex $K$ of the group presentation $\Pe=\langle x,y\,|\,x^{k+1}yxy \rangle$, with $k\geq1$ odd. Also, we keep considering the non-trivial local integer coefficient system $\varrho:\pi_1(\RP^2)\to{\rm Aut}(\Z)$.

Let us consider the cellular map $\omega:K\to\RP^2$ defined naturally by collapsing the $1$-cell $e^1_x$ to the $0$-cell $c^0$ of $\RP^2$. It is possible to understand the map $\omega$ by considering $K$ as the identification space obtained from the disc $D^2$ with identifications in its boundary $S^1=\partial D^2$ respect to the word $r=x^{k+1}yxy$. We explain: first we divide $S^1$ into $k+4$ oriented arc segments, all of the them with the counter-clockwise orientation, enumerated from a selected point $e^0$ by $x,\ldots,x,y,x,y$. Then we collapse the first $k+1$ arcs indexed with the letter $x$ to the point $e^0$ and we collapse the other arc indexed with the letter $x$ to another point, we say, $-e^0$. That way, we obtain a new disc whose boundary are composed by two oriented arcs, both with the counter-clockwise orientation, indexed by the letter $y$. Then, by the collage of these arcs one with other, we obtain the projective plane.

We have described the map $\omega$ in such a way that it is easy to see that the $1$-cell $e^1_y$ is identified with the $1$-cell $c^1$ and the interior of the $2$-cell $e^2$ is mapped homeomorphically onto the interior of the $2$-cell $c^2$.

Consider the induced homomorphism $\omega_{\#}:\Pi\rightarrow\pi_1(\RP^2)$ on fundamental groups. Of course, $\omega_{\#}(\bar{x})=1$ and $\omega_{\#}(\bar{y})=-1$. Hence, the map $\omega$ co-induces on $K$ the local integer coefficient system $\varrho\circ\omega_{\#}:\Pi\rightarrow{\rm Aut}(\Z)$ given by $\varrho\circ\omega_{\#}(\bar{x})=1$ and $\varrho\circ\omega_{\#}(\bar{y})=-1$, that is, $\varrho\circ\omega_{\#}=\beta$. It follows that $\omega$ induces a homomorphism $$\omega^{\ast}:H^2(\RP^2;_{\varrho}\!\Z)\to H^2(K;_{\beta}\!\Z).$$

\begin{proposition}\label{Proposition-omega-homo}
For each $k\geq3$ odd, the induced homomorphism $\omega^{\ast}:H^2(\RP^2;_{\varrho}\!\Z)\to H^2(K;_{\beta}\!\Z)$ corresponds to the natural epimorphism $\Z\to\Z/k\Z$.
\end{proposition}

The twisted cohomology group $H^2(\RP^2;_{\varrho}\!\Z)$ is  well known to be infinite cyclic. However, we present the computation in order to identify an explicit generator.  After that  we proceed to the proof of Proposition \ref{Proposition-omega-homo} itself.

Let us consider the projective plane $\RP^2$ as the model two complex of the group presentation $\Qe=\langle z\,|\,z^2\rangle$, so that, $\RP^2$ is endowed with its natural cellular structure $\RP^2=c^0\cup c^1\cup c^2$.

Let $\p:S^2\to\RP^2$ be the universal covering map and consider the sphere $S^2$ with its cellular structure co-induced by $\p$, so that, $S^2=s^0_1\cup s^0_2\cup s^1_1\cup s^1_2\cup s^2_1\cup s^2_2$, with $\p(s^i_j)=c^i$, for $1\leq i,j\leq 2$. As before, we choose the numeration of the cells so that the $1$-cell $s^1_1$ starts at $s^0_1$ and ends at $s^0_2=-s^0_1$, and $s^2_1$ is the $2$-cell whose orientation makes $\tilde{\partial}_2(s^2_1)=s^1_1+s^1_2$.

So, we take $s^0_1$, $s^1_1$ and $s^2_1$ to be the favourite cells of $S^2$, so that we have the following identifications of $\Z[\pi_1(\RP^2)]=\Z[\Z_2]$-module:
    \begin{align*}
    C_0(\widetilde{\RP^2})=C_0(S^2) & =\Z[\Z_2]\langle s^0_1 \rangle, \\ \noalign{\smallskip} C_1(\widetilde{\RP^2})=C_1(S^2) & =\Z[\Z_2]\langle s^1_1 \rangle, \\ \noalign{\smallskip} C_2(\widetilde{\RP^2})=C_2(S^2) & =\Z[\Z_2]\langle s^2_1\rangle.
    \end{align*}

Via this identifications and considering the action $\varrho:\Z_2\rightarrow{\rm Aut}(\Z)$, we have the corresponding twisted cellular chain complex of left $\Z[\Z_2]$-modules $$C_{\ast}^{\varrho}(S^2):0\rightarrow C_2(S^2)\stackrel{\tilde{\partial}_2^{\varrho}}{\longrightarrow}C_1(S^2)\stackrel{\tilde{\partial}_1^{\varrho}}{\longrightarrow}C_0(S^2)\rightarrow0,$$ in which the boundaries operators are given by
    \begin{align*}
    \tilde{\partial}_1^{\varrho}(s^1_1) & = \xi_{\varrho}(1-\bar{z})s^0_1=2s^0_1, \\ \noalign{\smallskip} \tilde{\partial}_2^{\varrho}(s^2_1) & =\xi_{\varrho}\big(\|\frac{\partial z^2}{\partial z}\|\big)s^1_1 = \xi_{\varrho}\big(\|1+z\|\big) s^1_1 = 0.
    \end{align*}

Consider the corresponding twisted cellular co-chain complex $$C^{\ast}_{\varrho}(S^2):0\leftarrow\hom^{\Z_2}(C_2(S^2);\Z)\stackrel{\tilde{\delta}_2^{\varrho}}{\longleftarrow}\hom^{\Z_2}(C_1(S^2);\Z)\stackrel{\tilde{\delta}_1^{\varrho}}{\longleftarrow}\hom^{\Z_2}(C_0(S^2);\Z)\leftarrow0.$$

In each $\hom^{\Z_2}(C_i(S^2);\Z)$, the integers $\Z$ is seen as a left $\Z[\Z_2]$-module via the action $\varrho:\Z_2\rightarrow{\rm Aut}(\Z)$. The co-boundaries operators $\tilde{\delta}_{\ast}^{\varrho}$ are defined by the dual form $\tilde{\delta}_{\ast}^{\varrho}(\phi)=\phi\circ\tilde{\partial}_{2}^{\varrho}$.

Particularly, since $\tilde{\partial}_2^{\varrho}=0$, also $\tilde{\delta}_{\ast}^{\varrho}=0$. Therefore, $$H^2(\RP^2;_{\varrho}\!\Z)=\hom^{\Z_2}(C_2(S^2);\Z)\approx\Z\langle \varphi_2^{\ast} \rangle,$$ in which the generator is the co-chain $\varphi_2^{\ast}:C_2(S^2)\to\Z$ defined by $\varphi_2^{\ast}(s^2_1)=1$, that is, $\varphi_2^{\ast}$ is the dual of the chain $s^2_1$.

\begin{proof}[Proof of Proposition \ref{Proposition-omega-homo}] Let $\q:\widetilde{K}\to K$ be the universal covering map and let $\tilde{\omega}:\tilde{K}\to S^2$ be the unique cellular map satisfying $\tilde{\omega}(\tilde{e}^0)=s^0_1$ and $\p\circ\tilde{\omega}=\omega\circ\q$, that is, $\tilde{\omega}$ is the lifting of $\omega$ to universal coverings, {\it starting at} $e^0_1$. Then $\tilde{\omega}$ collapses $\tilde{e}^1_x$ to $s^0_1$ and maps the interior of the cells $\tilde{e}^1_y$ and $\tilde{e}^2$ homeomorphically onto the interior of $s^1_1$ and $s^2_1$, respectively.

Let consider the following commutative diagram, in which the vertical arrows are the homomorphisms induced by $\tilde{\omega}$ in level of chains:
    \begin{center}
    \begin{tabular}{c}\xymatrix{0 \ar[r] & C_2(\tilde{K}) \ar[d]^-{\tilde{\omega}^2_{\#}} \ar[r]^-{\tilde{\partial}_2^{\beta}} & C_1(\tilde{K}) \ar[d]^-{\tilde{\omega}^1_{\#}} \ar[r]^-{\tilde{\partial}_1^{\beta}} & C_0(\tilde{K}) \ar[d]^-{\tilde{\omega}^0_{\#}} \ar[r] & 0 \\ 0 \ar[r] & C_2(S^2) \ar[r]^-{\tilde{\partial}_2^{\varrho}=0} & C_1(S^2) \ar[r]^-{^{\tilde{\partial}_1^{\varrho}}} & C_0(S^2) \ar[r] & 0}
    \end{tabular}
    \end{center}

By the definition of the map $\tilde{\omega}$, the homomorphism $\tilde{\omega}^0_{\#}$, $\tilde{\omega}^1_{\#}$ and $\tilde{\omega}^2_{\#}$ are given, in terms of its values on the generators of its domains, by: $$\tilde{\omega}^0_{\#}(\tilde{e}^0)=s^0_1 \quad{\rm and}\quad \tilde{\omega}^1_{\#}(\tilde{e}^1_x)=0, \ \tilde{\omega}^1_{\#}(\tilde{e}^1_y)=s^1_1 \quad{\rm and}\quad \tilde{\omega}^2_{\#}(\tilde{e}^2)=s^2_1.$$

Now, we consider the corresponding commutative diagram in level of co-chains. To shorten, we denote $C^i(X)={\rm Hom}^G(C_i(X);\Z)$, for $(X,G)=(\tilde{K},\Pi)$ and $(X,G)=(S^2,\Z_2)$.
    \begin{center}
    \begin{tabular}{c}\xymatrix{0 & C^2(\tilde{K}) \ar[l] & C^1(\tilde{K}) \ar[l]_-{\tilde{\delta}_2^{\beta}} & C^0(\tilde{K}) \ar[l]_-{\tilde{\delta}_1^{\beta}} & 0 \ar[l] \\ 0 & C^2(S^2) \ar[l] \ar[u]_-{\tilde{\omega}_2^{\#}} & C^1(S^2) \ar[l]_-{\tilde{\delta}_2^{\varrho}=0} \ar[u]_-{\tilde{\omega}_1^{\#}} & C^0(S^2) \ar[l]_-{\tilde{\delta}_1^{\varrho}} \ar[u]_-{\tilde{\omega}_0^{\#}} & 0 \ar[l]}
    \end{tabular}
    \end{center}

By duality, the homomorphisms $\tilde{\omega}_0^{\#}$, $\tilde{\omega}_1^{\#}$ and $\tilde{\omega}_2^{\#}$ are given, in terms of its values in the generators of its domains, by: $$\tilde{\omega}^0_{\#}(\varphi_0^{\ast})=\phi_0^{\ast} \quad{\rm and}\quad \tilde{\omega}_1^{\#}(\varphi_1^{\ast})=\phi_y^{\ast} \quad{\rm and}\quad \tilde{\omega}_2^{\#}(\varphi_2^{\ast})=\phi_2^{\ast}.$$

It follows that the homomorphism $\omega^{\ast}:H^2(\RP^2;_{\varrho}\!\Z)\to H^2(K;_{\beta}\!\Z)$ is given by: $$\frac{C^2(S^2)}{{\rm Im}(\tilde{\delta}_2^{\varrho})} \ni \varphi_2^{\ast}+0 \mapsto \tilde{\omega}_2^{\#}(\varphi_2^{\ast})+{\rm Im(\tilde{\delta}_2^{\beta})}=\varphi_2^{\ast}+\langle k\varphi_2^{\ast}\rangle\in \frac{C^2(\tilde{K})}{{\rm Im}(\tilde{\delta}_2^{\beta})}.$$

Therefore, $\omega^{\ast}$ corresponds to the natural epimorphism $\Z\to\Z/k\Z$.
\end{proof}

Obviously, we present more calculations than necessary to prove Proposition \ref{Proposition-omega-homo}. We do this in order to make more clear the induced homomorphisms by the maps involved.


\section{Proof of the Main Theorem}\label{Section-Proof}

\hspace{4mm} Via the natural identification ${\rm Aut}(\Z)\approx\Z_2$, we have $\hom(\Pi;\Z_2)=\{1,\beta\}$, in which $1$ denotes the trivial homomorphism and $\beta$ is the homomorphism given at the beginning of Section \ref{Section-Complex-K}. Thus, we have $$[K;\RP^2]^{\ast}=[K;\RP^2]_{1}^{\ast}\sqcup[K;\RP^2]_{\beta}^{\ast}.$$

Since $K$ is aspherical, we have the following bijections given in \cite[Theorem 4.2]{Livro-Verde-Chapter-II}: $$[K;\RP^2]_{1}^{\ast}\equiv H^2(K;\Z)=0 \qquad {\rm and} \qquad [K;\RP^2]_{\beta}^{\ast}\equiv H^2(K;_{\beta}\!\Z)\approx\Z/k\Z,$$

Hence, the sets $[K;\RP^2]_{1}\equiv[K;\RP^2]_{1}^{\ast}$ have only one element, namely, the (free or based) homotopy class of the constant map at the $0$-cell $c^0$. In the assertions (1) and (2) of Theorem \ref{Main-Theorem}, this element corresponds to $1$ and the component $[K;\RP^2]^{\ast}_{1}$ of $[K;\RP^2]^{\ast}$, as well as the component $[K;\RP^2]_{1}$ of $[K;\RP^2]$, corresponds to $\{1\}$. This proves a part of the assertions (1) and (2) of Theorem \ref{Main-Theorem}.

On the other hand, the set $[K;\RP^2]_{\beta}^{\ast}$ has $k$ elements, and we describe them and also the elements of the set $[K;\RP^2]_{\beta}\equiv[K;\RP^2]_{\beta}^{\ast}/\pi_1(\RP^2)$.

If $k=1$, then $[K;\RP^2]^{\ast}_{\beta}$ has only one element and so $[K;\RP^2]_{\beta}\equiv[K;\RP^2]^{\ast}_{\beta}$. In the statement of Theorem \ref{Main-Theorem}, assertion (1), this element corresponds to $\bar{0}$ and the component $[K;\RP^2]^{\ast}_{\beta}$ of $[K;\RP^2]^{\ast}$, as well as the component $[K;\RP^2]_{\beta}$ of $[K;\RP^2]$, corresponds to $\{\bar{0}\}$. Now, if $[f]^{\ast}\in[K;\RP^2]^{\ast}_{\beta}$, then $f_{\#}=\beta$ and, since $H^2(K;_{\beta}\!\Z)=0$, it follows from \cite[Theorem 1.1]{Fenille-One-Relator-RP2} that $f$ is homotopic to a non-surjective map.

We have completed the proof of the assertion (1) of Theorem \ref{Main-Theorem}.

In order to prove what is missing from assertion (2), we take $k=2p-1\geq3$. 

Let ${\rm Odd}(k)$ be the set of the odd integers in the set $\{2-k,\ldots,k-2\}$. Then ${\rm Odd}(k)$ has $k-1=2p-2$ elements, being $p-1$ of them positive and $p-1$ of them negative. 

For each $n\in{\rm Odd}(k)$ we define the map $$f_n=h_n\circ\omega:K\to\RP^2,$$ in which $\omega:K\to\RP^2$ is the map built in Section \ref{Section-Key-Map} and $h_n:\RP^2\to\RP^2$ is the based map whose twisted degree is $d_{\varrho}(h_n)=n$, as presented in Section \ref{Section-Self-Maps}. Additionally, define $$f_0=h_k\circ\omega:K\to\RP^2.$$

For each $n\in{\rm Odd}(k)\cup\{0\}$, the homomorphism $(f_n)_{\#}:\Pi\to\pi_1(\RP^2)$ is equal to $\beta$ and the homomorphism $f_n^{\ast}:H^2(\RP^2;_{\varrho}\!\Z)\to H^2(K;_{\beta}\!\Z)$ corresponds to $\mu_n:\Z\to\Z/k\Z$ given by $\mu_n(1)=n+k\Z$. It follows that, for each $n\in{\rm Odd}(k)$, the map $f_n$ is strongly surjective and, for $m\neq n$ in ${\rm Odd}(k)\cup\{0\}$, the maps $f_m$ and $f_n$ are not based homotopic. 
Therefore, we have $$[K;\RP^2]_{\beta}^{\ast}=\big\{[f_n]^{\ast} : n\in{\rm Odd}(k)\cup\{0\}\big\}\equiv\Z_k.$$

However, since the maps $h_n$ and $h_{-n}$ are freely homotopic, also the maps $f_n$ and $f_{-n}$ are freely homotopic, for each $n\in{\rm Odd}(k)$, that is, the based homotopy classes $[f_n]^{\ast}$ and $[f_{-n}]^{\ast}$ are exchanged by the action of $\pi_1(\RP^2)$ over $[K;\RP^2]^{\ast}_{\beta}$. On the other hand, since $k$ is odd, it is obvious that the remaining class $[f_0]^{\ast}$ is fixed by the action of $\pi_1(\RP^2)$.

Therefore, taking ${\rm Odd}^{+}(k)$ to be the set of the positive integers in ${\rm Odd}(k)$, we have $$\frac{[K;\RP^2]_{\beta}^{\ast}}{\pi_1(\RP^2)}\equiv\big\{[f_n] : n\in{\rm Odd}^+(k)\cup\{0\}\}\equiv\Z_p.$$ 

Since we have proved that each $f_n$, for $n\in{\rm Odd}^{+}(k)$, is strongly surjective, in order to finish the proof of Theorem \ref{Main-Theorem} it remains to prove that $f_0$ is not strongly surjective.

For this, consider the cellular map $g^1:K^1=S^1_x\vee S^1_y\to S^1$ which maps $S^1_x$ encircling $2$ times into $S^1$, in the positive orientation, and maps $S^1_y$ encircling $k+2$ times in $S^1$, in the opposite orientation. The induced homomorphism $g^1_{\#}:F(x,y)\approx\pi_1(K^1)\to\pi_1(S^1)\approx\Z$ is given by $g^1_{\#}(x)=2$ and $g^1_{\#}(y)=-(k+2)$. It follows that $g^1_{\#}(r)=0$, and so $g^1$ extends to a cellular map $\overline{g}:K\to S^1$. Composing $\overline{g}$ with the skeleton inclusion $l:S^1\hookrightarrow\RP^2$ we obtain a map $g:K\to\RP^2$ such that the induced homomorphism $g_{\#}:\Pi\to\pi_1(\RP^2)$ is equal to $\beta$ and the induced homomorphism $g^{\ast}:H^2(\RP^2;_{\varrho}\!\Z)\to H^2(K;_{\beta}\!\Z)$ is trivial. It follows from the one-to-one correspondence $[K;\RP^2]^{\ast}_{\beta}\equiv H^2(K;_{\beta}\!\Z)\approx\Z/k\Z$ that the based homotopy classes $[g]^{\ast}$ and $[f_0]^{\ast}$ are equal, since both correspond to the zero class in $H^2(K;_{\beta}\!\Z)$. Therefore, $f_0$ is homotopic to the non-surjective map $g$. \hspace{7.5cm} $\square$


\section*{Acknowledgments}

\hspace{4mm} Part of this work was develop during the visit of the second author to the Faculdade de Matem\'atica -- Universidade Federal de Uberl\^andia, during the period Oct31/Nov03, 2019. The second author would like thank the Faculdade de Matem\'atica for the great hospitality. This work is part of the Projeto Tem\'atico FAPESP: {\it Topologia Alg\'ebrica, Geom\'etrica e Diferencial} -- 2016/24707-4 (Brazil). 



\vspace{5mm}

\noindent{{\sc Marcio Colombo Fenille} \\ Faculdade de Matem\'atica -- Universidade Federal de Uberl\^andia \\ Av.\,Jo\~ao Naves de \'Avila 2121, Sta M\^onica, CEP:~38408-100, Uberl\^andia MG, Brazil} \\ E-mail address: {\sl mcfenille@gmail.com}

\vspace{3mm}

\noindent{{\sc Daciberg Lima Gon\c calves} \\ Department of Mathematics -- IME -- University of S\~ao Paulo \\ Rua~do~Mat\~ao 1010, CEP:~05508--090, S\~ao Paulo SP, Brasil} \\ E-mail address: {\sl dlgoncal@ime.usp.br}

\end{document}